\documentclass{article}
\usepackage{amsmath,amssymb,amsthm,enumerate,latexsym,comment,bm}
\usepackage[all]{xy}

\newtheorem{Lemma}{Lemma}[section]
\newtheorem{Proposition}[Lemma]{Proposition}

\theoremstyle{definition}

\newtheorem{Theorem}[Lemma]{Theorem}
\newtheorem{Definition}[Lemma]{Definition}
\newtheorem{Fact}[Lemma]{Fact}

\newtheorem{Remark}[Lemma]{Remark}
\newtheorem{Corollary}[Lemma]{Corollary}
\newtheorem*{Question*}{Question}

\newcommand\DistTo{\xrightarrow{
   \,\smash{\raisebox{-0.65ex}{\ensuremath{\scriptstyle\sim}}}\,}}

\newcommand{\D}{{\mathcal D}}

\newcommand{\g}{{\mathfrak g}}

\newcommand{\h}{{\mathfrak h}}
\newcommand{\q}{{\mathfrak q}}

\newcommand{\N}{\mathbb{N}}
\newcommand{\R}{\mathbb{R}}
\newcommand{\C}{\mathbb{C}}

\newcommand{\ve}{\varepsilon}

\newcommand{\ov}{\overline}

\newcommand{\p}{\partial}

\newcommand{\Tll}{T_{\lambda}^{l}}

\makeatletter
\@addtoreset{equation}{section} 
\makeatother

\title{Dimension of the space of intertwining operators from degenerate principal series representations}
\author{Taito Tauchi\thanks{Graduate School of Mathematical Sciences, the University of Tokyo, Meguro-ku, Tokyo, 153-8914, Japan, E-mail address: taito@ms.u-tokyo.ac.jp
}}
\date{}
\begin{document}
\maketitle
\begin{abstract}
Let $X$ be a homogeneous space of a real reductive Lie group $G$. It was proved by T. Kobayashi and T. Oshima that the regular representation $C^{\infty}(X)$ contains each irreducible representation of $G$ at most finitely many times if a minimal parabolic subgroup $P$ of $G$ has an open orbit in $X$, or equivalently, if the  number of $P$-orbits on $X$ is finite. In contrast to the minimal parabolic case, for a general parabolic subgroup $Q$ of $G$, we find a new example that the regular representation $C^{\infty}(X)$ contains degenerate principal series representations induced from $Q$ with infinite multiplicity even when the number of $Q$-orbits on $X$ is finite. 
\end{abstract}
{\bf MSC2010;} primary 22E46; secondary 22E45, 53C30.


\section{Introduction}
\label{Introduction}
Let $G$ be a real reductive algebraic Lie group, and $H$ an algebraic subgroup of $G$. T. Kobayashi and T. Oshima established the criterion of finite multiplicity for regular representations on $G/H$.
\begin{Fact}[{\cite[Theorem A]{KO}}]
\label{KO}
The following two conditions on the pair $(G,H)$ are equivalent:
\begin{enumerate}[(i)]
\item
$\dim {\rm Hom}_{G}
(\pi,C^{\infty}(G/H,\tau))<\infty$
for all $(\pi,\tau)\in \hat{G}_{{\rm smooth}}\times\hat{H}_{{\rm f}}$.
\item
$G/H$ is real spherical.
\end{enumerate}
\end{Fact}
Here $\hat{G}_{\rm{smooth}}$ denotes the set of equivalence classes of irreducible smooth admissible Fr\'echet representations of $G$ with 
moderate growth, and $\hat{H}_{{\rm f}}$ that of algebraic irreducible finite-dimensional representations of $H$. 
Given $\tau\in\hat{H}_{{\rm f}}$, 
we write $C^{\infty}(G/H,\tau)$ for the Fr\'echet space of smooth sections of the $G\mathchar`-$homogeneous vector bundle over $G/H$ associated to $\tau$.
The terminology {\it real sphericity} was introduced by T. Kobayashi \cite{RS} in his search of a broader framework for global analysis on homogeneous spaces than the usual (e.g., reductive symmetric spaces).
\begin{Definition}
A homogeneous space $G/H$ is {\it real spherical} if a minimal parabolic subgroup $P$ of $G$ has an open orbit in $G/H$.
\end{Definition}
The following equivalence is well known by the work of B. Kimelfeld \cite{Kimelfeld} and the real rank one reduction of T. Matsuki \cite{Matsuki}:
\begin{Fact}[{\cite[Theorem 2.2]{Bien}}]
\label{KM}
$G/H$ is real spherical if and only if the number of $H$-orbits on $G/P$ is finite. 
In other words, 
the condition (ii) in Fact \ref{KO} is equivalent to the following condition (iii):
\begin{enumerate}[(iii)]
\item
$\#(H\backslash G/P)<\infty$.
\end{enumerate}
\end{Fact}
Therefore, for a minimal parabolic $P$, the three conditions (i), (ii), and (iii) are equivalent by Fact \ref{KO} and Fact \ref{KM} (see Figure 1.1 below).  Then one might ask a question what will happen to the relationship among the three conditions, if we replace $P$ by a general parabolic subgroup $Q$ of $G$.
For this, we need to make a precise definition of variants of (i), (ii), and (iii) for a parabolic subgroup $Q$ of $G$.
\begin{Definition}[{\cite[Definition 6.6]{Kobshintani}}]
We say $\pi \in \hat{G}_{\rm smooth}$ belongs to $Q${\it -series} if $\pi$ occurs as a subquotient of the degenerate principal series representation $C^{\infty}(G/Q,\tau)$ for some $\tau \in \hat {Q}_{\rm f}$.
\end{Definition}
$
\hspace{-0.2cm}\quad\quad P:{\rm minimal\: parabolic}\quad\quad \quad\hspace{2.0cm} Q:{\rm general\:parabolic}\\
\hspace{1cm}
\xymatrix{
& (\rm{i}) \ar@{<=>}[ld]_{{\rm Fact \ref{KO}}} \ar@{<=>}[rd] & \\
(\rm{ii}) \ar@{<=>}[rr]_{{\rm Fact \ref{KM}}}& & (\rm{iii})
}\hspace{2cm}
$
$
\xymatrix{
& (\rm{i}_{Q}) \ar@<0.1cm>@{=>}[ld]|{\vspace{0.2cm}\hole\hspace{0.4cm}{\rm Fact.\ref{K}}}
 \ar@<0.1cm>@{<=}[rd]|{\hole\hspace{-0.2cm}{\rm No}}^{\hspace{0.25cm}{\rm Thm.\ref{Main}}} & \\
(\rm{ii}_{Q}) \ar@<0.2cm>@{=>}[ru]|{\hole\hspace{-0.3cm}{\rm No}}^{{\rm Thm.\ref{Main}}\hspace{0.2cm}} 
\ar@<0.15cm>@{<=}[rr]& &
 (\rm{iii}_{Q}) \ar@<0.1cm>@{<=}[ll]|{\hole\hspace{-0.2cm}{\rm No}}
}
$
\begin{center}\hspace{0cm}Figure 1.1\hspace{5.0cm}Figure 1.2\end{center}
We set $\hat{G}_{{\rm smooth}}^{Q}:=\{\pi\in\hat{G}_{\rm smooth}\mid \pi {\rm \: belongs\: to\:}
Q\mathchar`-{\rm series.}\}$.
Obviously, $\hat{G}_{{\rm smooth}}^{Q}\supset \hat{G}_{{\rm smooth}}^{Q'}$ if $Q\subset Q'$.
Moreover,
$\hat{G}^{Q}_{\rm smooth}$ is equal to $\hat{G}_{\rm smooth}$ if $Q=P$ (minimal parabolic) by Harish-Chandra's subquotient theorem \cite{subquo} and to $\hat{G}_{{\rm f}}$ if $Q=G$.
\begin{Definition}
For a parabolic subgroup $Q$ of $G$, we define the three conditions (i$_{Q}$), (ii$_{Q}$), and (iii$_{Q}$), respectively, as follows:
\begin{enumerate}[(i$_{Q}$)]
\item
 $\dim {\rm Hom}_{G}
(\pi,C^{\infty}(G/H,\tau))<\infty$
for all $(\pi,\tau)\in \hat{G}^{Q}_{{\rm smooth}}\times\hat{H}_{{\rm f}}$.
\item
$Q$ has an open orbit in $G/H$.
\item
$\#(H\backslash G/Q)<\infty$.
\end{enumerate}
\end{Definition}
The conditions (i$_{Q}$), (ii$_{Q}$), and (iii$_{Q}$) reduce to (i), (ii), and (iii), respectively, if $Q=P$ (minimal parabolic), and we know from Fact \ref{KO} and Fact \ref{KM} (see also Figure 1.1) that the following equivalences hold:
\begin{align*}
({\rm i}_{Q}) \iff ({\rm ii}_{Q}) \iff ({\rm iii}_{Q})\quad&{\rm if}\quad Q=P.\\
\intertext{Further, it is obvious from the Frobenius reciprocity that the condition (i$_{Q}$) automatically holds if $Q=G$; (ii$_{Q}$) and (iii$_{Q}$) obviously hold. Hence}
({\rm i}_{Q}) \iff ({\rm ii}_{Q}) \iff ({\rm iii}_{Q})\quad& {\rm if}\quad Q=G.
\end{align*}
In the general setting, clearly, (iii$_{Q}$) implies (ii$_{Q}$), however the converse may fail if $Q$ is not a minimal parabolic subgroup of $G$.
On the other hand, the implication (i$_{Q}$) $\Rightarrow$ (ii$_{Q}$) is true. In fact, the following stronger theorem holds:
\begin{Fact}[{\cite[Corollary 6.8]{Kobshintani}}]
\label{K}
If there exists $\tau \in \hat{H}_{\rm f}$ such that for all $\pi \in \hat{G}^{Q}_{\rm smooth}$ $\dim {\rm Hom}_{G}(\pi, C^{\infty}(G/H,\tau))
<\infty$, then (ii$_{Q}$) holds.
\end{Fact}
An open problem is whether the converse statement holds or not.
\begin{Question*}
Does
the finite-multiplicity condition (i$_{Q})$ in representation theory
follows from 
the geometric condition (ii$_{Q}$) (or (iii$_{Q}))$?
\end{Question*}
We give a negative answer to this question in this paper. Explicitly, we prove the theorem below:
\begin{Theorem}
\label{Main}
Let $Q$ be a maximal parabolic subgroup of $G=SL(2n,\R)$ such that
$G/Q$ is isomorphic to the real projective space $\R\mathbb{P}^{2n-1}$.
Then if $n\geq 2$, there exists an algebraic subgroup $H$ of $G$ satisfying the following two conditions:
\begin{enumerate}[1)]
\item
$\#(H\backslash G/Q)<\infty$,
\item
$\dim {\rm Hom}_{G}(C^{\infty}(G/Q,\chi),C^{\infty}(G/H))=\infty
$
for some one-dimensional representation $\chi$ of $Q$.
\end{enumerate}
Furthermore, if $n\geq 3$, $H$ satisfies the following condition:
\begin{enumerate}[2')]
\item
$\dim {\rm Hom}_{G}(C^{\infty}(G/Q,\chi),C^{\infty}(G/H))=\infty
$
for any one-dimensional representation $\chi$ of $Q$.
\end{enumerate}
\end{Theorem}
We summarize the relationship among the conditions (i$_{Q}$), (ii$_{Q}$), and (iii$_{Q}$) as follows: (i$_{Q}$) $\Rightarrow$ (ii$_{Q}$) is true by Fact \ref{K}. Theorem \ref{Main} implies that neither (iii$_{Q}$) $\Rightarrow$ (i$_{Q}$) nor (ii$_{Q}$) $\Rightarrow$ (i$_{Q}$) holds, see Figure 1.2. 
\begin{Remark}
The recent paper \cite[Theorem D]{G} claimed the following: {\it
Suppose that a real algebraic group $H$ acts on a real  algebraic smooth variety $M$ with $\#(H\backslash M)<\infty$ and that $E$ is an algebraic $H$-homogeneous vector bundle on $M$. Then, for any $n\in\N$,
\begin{eqnarray}
\sup_{\tau\in\hat{H}_{\rm f}\atop
\dim\tau=n}
\dim {\rm Hom}_{H}(\tau, {\mathcal S}^{*}(M,E))
<\infty.
\label{S}
\end{eqnarray}
}
We note that ${\mathcal S}^{*}(M,E)$ 
can be identified with the space
$\D'(M)$ of distributions in the case that $M$ is compact 
and $E$ is the trivial bundle $M\times\C$ \cite[Chapter 1.5]{SN}.
Therefore $(\ref{S})$ would imply 
\begin{eqnarray}
\dim {\rm Hom}_{H}({\bf 1},\D'(M))=\dim \D'(M)^{H}<\infty,
\label{Gou}
\end{eqnarray}
when $\#(H\backslash M)<\infty$ and $M$ is compact. Here ${\bf 1}$ denotes the trivial
one-dimensional representation of $H$.

However, one sees from Fact \ref{GPGP} that (\ref{Gou}) contradicts to Theorem \ref{Main},
when applied to $M=\R{\mathbb P}^{2n-1}$. 
Thus Theorem \ref{Main} is a counterexample to \cite[Theorem D]{G}.
Indeed, it seems to the author that a gap in the proof of \cite[Theorem D]{G}
comes from a false statement 
$\#(H\backslash G/Q)<\infty \Rightarrow \#(H_{\C}\backslash G_{\C}/Q_{\C})<\infty$,
see Remark \ref{holonomic} below.
\end{Remark}
The outline of this article as follows: 
In Section \ref{Reduction to distribution},
we recall some general facts concerning distribution kernels,
which were proved by T. Kobayashi and B. Speh \cite{KS}. 
In Section \ref{Preliminary}, we fix some basic notation for distributions on the complex Euclidean space.
In Section \ref{Construction}, we construct the subgroup $H$ of $G$ and give a proof of Theorem \ref{Main}.
\section{Reduction to distribution kernels}
\label{Reduction to distribution}
In this section, we reformulate the condition 2) of Theorem \ref{Main} by means of distribution kernels using Fact \ref{GPGP} below.
\begin{Definition}
Let $G$ be a real Lie group and $H$ a closed subgroup of $G$.  For $\tau \in \hat{H}_{{\rm f}}$, we define the finite-dimensional representation of $H$ by $\tau_{2\rho}^{\vee}:=\tau^{\vee}\otimes\C_{2\rho}$ where $\tau^{\vee}$ is the contragredient representation of $\tau$ and
$\C_{2\rho}$ denotes the one-dimensional representation of $H$ given by $h\mapsto \left|\: \det ({\rm Ad}(h):\g/\h\to\g/\h)\right|^{-1}$.\end{Definition}
\begin{Fact}[{\cite[Proposition\:3.2]{KS}}]
\label{GPGP}
Let $G$ be a real Lie group. Suppose that $G'$ and $H$ are closed subgroups of $G$ and that 
$H'$ is a closed subgroup of $G'$. Let $\tau$ and $\tau'$ be finite-dimensional representations of $H$ and $H'$, respectively.
\begin{enumerate}[(1)]
\item
There is a natural injective map:
\begin{eqnarray}
{\rm Hom}_{G'}\left(C^{\infty}(G/H,\tau),C^{\infty}(G'/H',\tau')\right)
\hookrightarrow
\left(\D'(G/H,\tau_{2\rho}^{\vee})\otimes \tau'\right)^{H'}.
\label{surj}
\end{eqnarray}
Here $\left(\D'(G/H,\tau_{2\rho}^{\vee})\otimes \tau'\right)^{H'}$ denotes 
the space of $H'$-fixed vectors under the diagonal action.
\item
If $H$ is cocompact in $G$ (e.g., a parabolic subgroup of $G$ or a uniform lattice), then (\ref{surj}) is a bijection.
\end{enumerate}
\end{Fact}
We apply this fact to the setting of Theorem \ref{Main}.
Recall that $G=SL(2n,\R)$ and $Q$ is a maximal parabolic subgroup of $G$
such that $G/Q\simeq \R{\mathbb P}^{2n-1}$.
For $\lambda\in\C$,
we define a one-dimensional representation $\chi_{\lambda}:Q\to GL(1,\C)$ by
$g\mapsto |\det( {\rm Ad}(g):\g/\q\to\g/\q)|^{\frac{-\lambda}{2n}}$.
We denote by $\D'(\R^{2n}\backslash\{0\})_{even,\lambda-2n}$  the space of even homogeneous distributions of degree $\lambda-2n$ on $\R^{2n}\backslash\{0\}$.
\begin{Corollary}
\label{H}
For any closed subgroup $H$ of $G$, we have
$${\rm Hom}_{G}(C^{\infty}(G/Q,\chi_{\lambda}),C^{\infty}(G/H))\simeq\D'(\R^{2n}\backslash\{0\})^{H}_{even,\lambda-2n}.$$
\end{Corollary}
\begin{proof}
This follows from Fact \ref{GPGP} because 
$\C_{2\rho}=\chi_{2n}$ as representations of $Q$ and
$\D'(G/Q,\chi_{\lambda})\simeq \D'(\R^{2n}\backslash\{0\})_{even,-\lambda}$ in the setting of Corollary \ref{H}.
\end{proof}
\section{Notation for distributions on the complex Euclidean space}
\label{Preliminary}
In Section \ref{Construction}, 
we shall consider a linear group action on $\C^{n}$
regarded as a {\it real} vector space.
In order to avoid possible confusion, we
prepare 
some notation for distributions on the complex Euclidean space $\C^{n}$ regarded as a real vector space.
Identifying $\C^{n}$ with $\R^{2n}$ by $z=(z_{1},\dots,z_{n})=(x_{1}+iy_{1},\dots,x_{n}+iy_{n})$, we  write $\D(\C^{n})$ and $\D'(\C^{n})$ for the spaces of 
$C^{\infty}$ functions with compact support and distributions on $\C^{n}\simeq\R^{2n}$, respectively.
We define a distribution $\delta(z_{n},\ov{z}_{n})\in\D'(\C^{n})\simeq\D'(\R^{2n})$ by
\begin{eqnarray*}
\delta(z_{n},\ov{z}_{n})(\phi)&:=&
\frac{1}{(-2i)^{n}}
\int_{\C^{n-1}} \phi(z_{1},\dots,z_{n-1},0)\:dz_{1}d\ov{z}_{1}\dots dz_{n-1}d\ov{z}_{n-1}\\
&=&\frac{1}{-2i}\int_{\R^{2n-2}}\phi(x'+iy',0)\:dx_{1}dy_{1}\dots dx_{n-1}dy_{n-1}
\end{eqnarray*}
for every test function $\phi\in \D(\C^{n})\simeq \D(\R^{2n})$
where $x'+iy':=(x_{1}+iy_{1},\dots,x_{n-1}+iy_{n-1})$.
We write $\delta(\cdot)$ for the usual Dirac delta function on $\R$
and regard it as a distribution on $\R^{2n}$ by the pull-back via
the projection $\R^{2n}\to\R$. Then we have 
\begin{eqnarray}
\delta(z_{n},\ov{z}_{n})=(-2i)^{-1}\delta(x_{n})\delta(y_{n})
\label{zxy}
\end{eqnarray}
as distributions on $\C^{n}\simeq \R^{2n}$.
Since the multiplication by $x_{n}$ or $y_{n}$ kills (\ref{zxy}),
so does it by $z_{n}$ or $\ov{z}_{n}=x_{n}-iy_{n}$, that is, 
\begin{eqnarray}
z_{n}\delta(z_{n},\ov{z}_{n})=\ov{z}_{n}\delta(z_{n},\ov{z}_{n})=0.
\label{zovz}
\end{eqnarray}
We define differential operators on $\C^{n}\simeq \R^{2n}$ by
$$
\frac{\p}{\p z_{j}}:=\frac{1}{2}\left(\frac{\p}{\p x_{j}}-i\frac{\p}{\p y_{j}}\right)
,\quad
\frac{\p}{\p \ov{z}_{j}}:=\frac{1}{2}\left(\frac{\p}{\p x_{j}}+i\frac{\p}{\p y_{j}}\right)
\quad(1\leq j\leq n).
$$
Multiplication of $\frac{\p^{l}}{\p z_{n}^{l}}\delta(z_{n},\ov{z}_{n})$
by distributions of $z_{1},\ov{z}_{1},\dots,z_{n-1},\ov{z}_{n-1}$ makes sense.
We note that a finite family $\{T_{l}\}_{l=1}^{m}$ of distributions on $\C^{n-1}
\backslash\{0\}$ vanish 
if the following equality as distributions on $\C^{n}\backslash\{0\}\simeq \R^{2n}\backslash\{0\}$ holds:
\begin{eqnarray}
\sum_{l=1}^{m}T_{l}(z_{1},\dots,z_{n-1})\frac{\p^{l}}{\p z_{n}^{l}}\delta(z_{n},\ov{z}_{n})=0.
\label{T=0}
\end{eqnarray}
Suppose a group $G$ acts linearly on $\C^{n}$ regarded as a real vector space.
In turn, $G$ acts on the spaces of $C^{\infty}$ functions $f$, distributions $T$, and differential operators $D$ on $\C^{n}\simeq \R^{2n}$.
We shall denote these actions by
\begin{eqnarray*}
(g\cdot f)(z)&:=&f(g^{-1}\cdot z),\\
(g\cdot T)(\phi)&:=&T(g^{-1}\cdot \phi),\\
(g\cdot D)(f)&:=&g\cdot(D(g^{-1}\cdot f)),
\end{eqnarray*}
where $g\in G, z\in\C^{n}$, and $\phi \in \D(\C^{n})\simeq \D(\R^{2n})$.
\section{Proof of Theorem \ref{Main}}
\label{Construction}
In this section, we take $G$ to be $SL(2n,\R)$, and construct an algebraic subgroup $H$ satisfying the two conditions 1) and 2) in Theorem \ref{Main}.
We begin with a $4$-dimensional $\R$-algebra $R_{\ve}$ defined by 
\begin{align}
R_{\ve}&:=\C\oplus \C\ve&\quad& {\rm as\: a\: vector\: space}, \nonumber\\
(a+b\ve)(c+d\ve)&:=(ac+b\ov{d})+(b\ov{c}+ad)\ve&\quad&{\rm as\: a\: ring},
\label{mult}
\end{align}
with $\ve$ being just a symbol, and $a,b,c,d\in\C$.
Regarding $\C$ as an $\R$-vector space, 
we let $R_{\ve}$ act $\R$-linearly on $\C$ by
\begin{eqnarray}
(a+b\ve)\cdot z:=az+b\ov{z}\qquad(a+b\ve\in R_{\ve},\:z\in\C).
\label{action}
\end{eqnarray}
\begin{Remark}
We write $i$ for the imaginary unit of $\C$, then by $(\ref{mult})$ we have
$$\ve^{2}=1,\quad i^{2}=-1,\quad i\ve=-\ve i.$$
Therefore $R_{\ve}$ is isomorphic to the real Clifford algebra $C(1,1)$ as an $\R$-algebra.
Hence we have $R_{\ve}\simeq C(1,1)\simeq M_{2}(\R)$ (for 
example,
\cite[Proposition 4.4.1]{CA}).
\end{Remark}
Let $M_{n}(R_{\ve})$ be the $\R$-algebra of all $n\times n$ matrices over $R_{\ve}$.
The left multiplication defines a (real) representation of $M_{n}(R_{\ve})$ on $\C^{n}$ regarded as
a vector space over $\R$.
This representation induces an injective $\R$-algebra homomorphism 
\begin{eqnarray}
\iota:M_{n}(R_{\ve})\hookrightarrow M_{2n}(\R),
\label{iota}
\end{eqnarray}
which is also surjective 
because the real dimensions of $M_{n}(R_{\ve})$ and $M_{2n}(\R)$ are the same. 
We define a subgroup $H$ of $M_{n}(R_{\ve})$ by
\begin{eqnarray}
H:=\left\{
h^{\theta}(\bm{a}):=\left( 
 \begin{array}{cccccc}
 e^{i\theta}&a_{1}\ve&a_{2}\ve^{2}&\cdots &a_{n-1}\ve^{n-1}\\
 &e^{i\theta}&a_{1}\ve&\ddots &\vdots\\
 &&e^{i\theta}&\ddots &a_{2}\ve^{2}\\
 &&&\ddots&a_{1}\ve\vspace{0.2cm}\\
 &&&&e^{i\theta}
 \end{array}
\right)
\:\middle|\:
\begin{array}{c}
\hspace{-0.2cm}\theta \in \R \\
\hspace{-0.2cm}\bm{a}\in \C^{n-1}\hspace{-0.3cm}
\end{array}
\right\},
\label{Hgroup}
\end{eqnarray}
where $\bm{a}=(a_{1},\dots,a_{n-1})\in \C^{n-1}$.
Then $\iota(H)$ is a subgroup of $GL(2n,\R)$.
\begin{Lemma}
\label{det}
$\det(\iota(H))=\{1\}$.
\end{Lemma}
\begin{proof}
For any $\bm{a}\in\C^{n-1}$,
it is clear that $\det\left(\iota\left(h^{0}(\bm{a})\right)\right)=1$
since $\iota(h^{0}(a))\in GL(2n,\R)$ is a unipotent matrix.
Moreover dividing $\iota\left(h^{\theta}(0,\dots,0)\right)\in GL(2n,\R)$ into $2\times 2$ block matrices,
we have $\det\left(\iota\left(h^{\theta}(0,\dots,0)\right)\right)=1$ 
for any $\theta\in\R$
because $e^{i\theta}$ acts on $\C\simeq \R^{2}$ as rotation.
Since the group $H$ is generated by elements of the form $h^{0}(\bm{a})$ and $h^{\theta}(0,\dots,0)$, the lemma is proved.
\end{proof}
By Lemma \ref{det}, we may identify $H$ in $M_{n}(R_{\ve})$ with $\iota(H)$ in $G=SL(2n,\R)$ via $\iota$.

The following proposition shows that the subgroup $H$ of $G$ satisfies the condition 1) in Theorem \ref{Main}.
\begin{Proposition}
\label{finiteorbit}
For every $j\in\{1,2,\dots,n\}$, 
there exists exactly one $H$-orbit on $G/Q$
of real dimension $2j-1$. 
These orbits exhaust all $H$-orbits on $G/Q$.
In particular, $\#(H\backslash G/Q)=n<\infty$.
\end{Proposition}
\begin{proof}
Let $\R^{\times}:=GL(1,\R)$ act on $\C^{n}$ by scalar multiplication and put $X:=(\C^{n}\backslash\{0\})/\R^{\times}$.
Identifying $\C^{n}$ with $\R^{2n}$, we have $X\simeq\R\mathbb{P}^{2n-1}\simeq G/Q$ 
and these isomorphisms induce a bijection:
\begin{eqnarray}
H\backslash X\simeq H\backslash G/Q.
\label{XGQ}
\end{eqnarray}
For $j\in\{1,2,\dots,n\}$, we define a real $(2j-1)$-dimensional submanifold of $X$ by
\begin{eqnarray}
Y_{2j-1}:=\{(z_{1},\dots,z_{j},0,\dots,0)\in \C^{n}\mid z_{j}\neq0\}/\R^{\times}\quad\subset\quad X.
\label{Y}
\end{eqnarray}
Then the group $H$ leaves $Y_{2j-1}$ invariant, and in fact it acts transitively.
Thus we have an orbit decomposition
$$
H\backslash X=
\bigcup_{j=1}^{n} Y_{2j-1}.
$$
Therefore $\#(H\backslash G/Q)=\#(H\backslash X)=n<\infty$.
\end{proof}
Let us prove that the subgroup $H$ of $G$ satisfies the condition 2') of Theorem \ref{Main} in the case of $n\geq 3$.
We define two real analytic vector fields $D$ and $\ov{D}$ on 
$\C^{n}\simeq\R^{2n}$ for $n\geq 3$ by
\begin{eqnarray}
D:=\ov{z}_{n-2}\frac{\p}{\p \ov{z}_{n-1}}+z_{n-1}\frac{\p}{\p z_{n}},\quad
\ov{D}:=z_{n-2}\frac{\p}{\p z_{n-1}}+\ov{z}_{n-1}\frac{\p}{\p \ov{z}_{n}}.
\label{DD}
\end{eqnarray}
For $l\in\N$, we define nonzero two distributions
$T_{\lambda}^{l},\ov{T}_{\lambda}^{l}\in\D'(\C^{n}\backslash\{0\})$
with holomorphic parameter $\lambda\in\C$
by
\begin{eqnarray}
T^{l}_{\lambda}(z):=\frac{1}{\Gamma\left(2-\frac{\lambda}{2}\right)}
D^{l}
\left(
|z_{n-1}|^{2-\lambda}
\delta(z_{n},\ov{z}_{n})
\right),
\label{TT1}
\end{eqnarray}
\begin{eqnarray}
\ov{T}^{l}_{\lambda}(z):=\frac{1}{\Gamma\left(2-\frac{\lambda}{2}\right)}
\ov{D}^{l}
\left(
|z_{n-1}|^{2-\lambda}
\delta(z_{n},\ov{z}_{n})
\right),
\label{TT2}
\end{eqnarray}
where $\Gamma(\cdot)$ denotes the gamma function.
We note that $|z_{n-1}|^{2-\lambda}=(x_{n-1}^{2}+y_{n-1}^{2})^{1-\frac{\lambda}{2}}$
has a simple pole at $\lambda\in2\N+4$ as a distribution and $\Gamma(2-\frac{\lambda}{2})$ has a simple pole at $\lambda\in2\N+4$.
Therefore $\Tll$ and $\ov{T}_{\lambda}^{l}$ define distributions with holomorphic parameter 
$\lambda \in \C$ (for example, see \cite[Appendix B1.4]{gelfand}).
Moreover $T_{\lambda}^{l}$ and $\ov{T}^{l}_{\lambda}$ are homogeneous distributions of degree $-\lambda$ because $|z_{n-1}|^{2-\lambda}$ and $\delta(z_{n},\ov{z}_{n})$
are homogeneous of degree $2-\lambda$ and $-2$, respectively,
and the operators $D$ and $\ov{D}$ preserve the degrees.
Clearly, $T^{l}_{\lambda}$ and $\ov{T}^{l}_{\lambda}$ are even distributions,
therefore $T^{l}_{\lambda},\ov{T}^{l}_{\lambda}\in\D'(\C^{n}\backslash\{0\})_{even,-\lambda}
\simeq \D'(G/Q,\chi_{\lambda})$.
\begin{Proposition}
\label{construction}
Suppose $n\geq 3$. 
Then for any $\lambda\in\C$ and any $l\in\N$, 
the distributions $T_{\lambda}^{l}$ and $\ov{T}_{\lambda}^{l}$ are $H$-invariant, that is, $T_{\lambda}^{l},\ov{T}_{\lambda}^{l}\in\D'(\C^{n}\backslash\{0\})^{H}_{even,-\lambda}$.\end{Proposition}
\begin{proof}
We prove only the claim for $T^{l}_{\lambda}$ 
as that for $\ov{T}^{l}_{\lambda}$ can be shown similarly.
We define elements of $H$ by the equality
\begin{eqnarray}
h(\theta):=h^{\theta}(0,\dots,0),\quad h_{j}(a):=h^{0}(0,\dots,0,\stackrel{j}{\stackrel{\vee}{a}},0,\dots,0),
\label{hj}
\end{eqnarray}
where $\theta\in\R$, $a\in\C$ and, $j\in\{1,2,\dots,n-1\}$ (see $(\ref{Hgroup})$ for notation). 
Then it is sufficient to prove that
$h(\theta)\cdot T^{l}_{\lambda}=T^{l}_{\lambda}$ 
for any $\theta\in\R$
and
$h_{j}(a)\cdot T^{l}_{\lambda}=T^{l}_{\lambda}$
for any $a\in\C$ and $j\in\{1,2,\dots,n-1\}$
because  the group $H$ is generated by elements of the form $h(\theta)$ and $h_{j}(a)$.
The first claim
follows easily from $h(\theta)\cdot z=e^{i\theta}z$ for $z\in\C^{n}$.
For the case of $j=1$ of the second claim,
we need the following:
\begin{Lemma}
\label{D}
Let $D$ be the vector field defined in (\ref{DD}). Then, we have
\begin{eqnarray*}
h_{1}(a)\cdot D&=&D+a\left(\ov{z}_{n-2}-\ov{a}z_{n-1}+|a|^{2}\ov{z}_{n}\right)\frac{\p}{\p z_{n-2}}
-a\ov{z}_{n}\frac{\p}{\p z_{n}}\quad (a\in\C).
\end{eqnarray*}
\end{Lemma}
This is an easy calculation, hence we omit the proof.

By Lemma \ref{D},
the following equality as distributions on $\C^{n}\backslash\{0\}\simeq\R^{2n}\backslash\{0\}$ holds:
\begin{eqnarray*}
(h_{1}(a)\cdot\Tll)(z)&=&\frac{1}{\Gamma\left(2-\frac{\lambda}{2}\right)}\left(
h_{1}(a)\cdot D
\right)^{l}
\left(
|z_{n-1}-a\ov{z}_{n}|^{2-\lambda}
\delta(z_{n},\ov{z}_{n})
\right)\\
&=&\frac{1}{\Gamma\left(2-\frac{\lambda}{2}\right)}
\left(
D-a\ov{z}_{n}\frac{\p}{\p z_{n}}
\right)^{l}
\left(
|z_{n-1}|^{2-\lambda}\delta(z_{n},\ov{z}_{n})
\right)\\
&=&\frac{1}{\Gamma\left(2-\frac{\lambda}{2}\right)}
D^{l}
\left(|z_{n-1}|^{2-\lambda}\delta(z_{n},\ov{z}_{n})
\right)\\
&=&\Tll(z).
\end{eqnarray*}
We have used (\ref{zovz}) and $\frac{\p}{\p z_{n-2}}\left(|z_{n-1}|^{2-\lambda}\delta(z_{n},\ov{z}_{n})\right)=0$ in the second equality.
For $j\in\{2,3,\dots,n-1\}$, 
$h_{j}(a)\cdot T^{l}_{\lambda}=T^{l}_{\lambda}$ can be shown similarly
in the case $j=1$.
Therefore $T_{\lambda}^{l}$ is $H$-invariant.
Thus the proof of proposition completes.
\end{proof}
\begin{Proposition}
\label{Dinfty}
If $n\geq 3$, for any $\lambda\in\C$ we have
$$\dim \D'(\C^{n}\backslash\{0\})^{H}_{even,-\lambda}=\infty.$$
\end{Proposition}
\begin{proof}
We know from Proposition \ref{construction} that $T^{l}_{\lambda}\in\D'(\C^{n}\backslash\{0\})^{H}_{even,-\lambda}$ for all $l\in\N$.
Therefore it is sufficient to prove that $\{\Tll\}_{l\in\N}$ is linearly independent. 
But this is a consequence of (\ref{T=0}) and the following equality as distributions on 
$\C^{n}\backslash\{0\}\simeq \R^{2n}\backslash\{0\}$:
\begin{eqnarray*}
T^{l}_{\lambda}(z)\hspace{-0.3cm}&=&\hspace{-0.3cm}
\frac{1}{\Gamma\left(2-\frac{\lambda}{2}\right)}
\left(\ov{z}_{n-2}\frac{\p}{\p \ov{z}_{n-1}}+z_{n-1}\frac{\p}{\p z_{n}}\right)^{l}
\left(
|z_{n-1}|^{2-\lambda}
\delta(z_{n},\ov{z}_{n})
\right)\\
&=&\hspace{-0.3cm}
\frac{1}{\Gamma\left(2-\frac{\lambda}{2}\right)}
\sum_{k=0}^{l}
\binom{l}{k}
\left( \ov{z}_{n-2}\frac{\p}{\p \ov{z}_{n-1}}\right)^{k}
\left(z_{n-1}\frac{\p}{\p z_{n}}\right)^{l-k}
\hspace{-0.1cm}\left(
|z_{n-1}|^{2-\lambda}
\delta(z_{n},\ov{z}_{n})
\right)\\
&=&\hspace{-0.3cm}
\frac{1}{\Gamma\left(2-\frac{\lambda}{2}\right)}
\sum_{k=0}^{l}
\binom{l}{k}
\left(
 \ov{z}_{n-2}^{k}z_{n-1}^{l-k}
 \frac{\p^{k} |z_{n-1}|^{2-\lambda}}{\p \ov{z}_{n-1}^{k}}
 \right)
\frac{\p^{l-k}}{\p z_{n}^{l-k}}
\delta(z_{n},\ov{z}_{n}).
\end{eqnarray*}
We have used the binomial expansion in the second equality.
\end{proof}
\begin{proof}[Proof of Theorem {\rm \ref{Main}} in the case $n\geq 3$]
We take $H$ to be the subgroup (\ref{Hgroup}) via the inclusion $\iota$ (\ref{iota}).
Then $H$ satisfies 1) by Proposition \ref{finiteorbit}. 
Moreover $H$ satisfies 2') by Corollary \ref{H} and Proposition \ref{Dinfty}
because $\D'(\R^{2n}\backslash\{0\})^{H}_{even,-\lambda}\simeq\D'(\C^{n}\backslash\{0\})^{H}_{even,-\lambda}$.
We note that any one-dimensional representation $\chi$ of $Q$
is of the form $\chi_{\lambda}$ for some
$\lambda\in\C$.
\end{proof}

Next we discuss in the case $n=2$.
For $\lambda=2$ in (\ref{TT1}) and (\ref{TT2}), the binomial expansion shows
\begin{eqnarray}
T^{l}_{2}(z)&=&
\left(
\ov{z}_{n-2}\frac{\p}{\p \ov{z}_{n-1}}+z_{n-1}\frac{\p}{\p z_{n}}
\right)^{l}
\delta(z_{n},\ov{z}_{n})\nonumber\\
&=&
\left(
z_{n-1}\frac{\p}{\p z_{n}}
\right)^{l}
\delta(z_{n},\ov{z}_{n}),
\label{Tl2}\\
\ov{T}^{l}_{2}(z)
&=&
\left(
\ov{z}_{n-1}\frac{\p}{\p \ov{z}_{n}}
\right)^{l}
\delta(z_{n},\ov{z}_{n}).
\label{Tl22}
\end{eqnarray}
In the second equality, we have used $\frac{\p}{\p \ov{z}_{n-1}}\delta(z_{n},\ov{z}_{n})=0$ because 
$\delta(z_{n},\ov{z}_{n})$ does not depend on the variable $\ov{z}_{n-1}$.
Then we define $T^{l}_{2}$ and $\ov{T}^{l}_{2}$ in the case of $(n,\lambda)=(2,2)$ by (\ref{Tl2}) and (\ref{Tl22}), respectively, in which the variables $z_{n-2},\ov{z}_{n-2}$ do not appear.
By using these distributions,
we prove the case $n=2$ of Theorem \ref{Main}.

\begin{proof}[Proof of Theorem \ref{Main} in the case of $n=2$]
We take $H$ to be the subgroup (\ref{Hgroup}) via the inclusion $\iota$ (\ref{iota}) as in the case of $n\geq 3$, then 
$H$ satisfies 1) by Proposition \ref{finiteorbit}.
Set $D':=z_{1}\frac{\p}{\p z_{2}}$. By (\ref{Tl2}) we have 
\begin{eqnarray*}
T^{l}_{2}(z)=\left(D'\right)^{l}
\delta(z_{2},\ov{z}_{2}).
\end{eqnarray*}
We note that the group $H$ is generated by elements of the form $h(\theta)$ and $h_{1}(a)$ in the case of $n=2$.
Just like before,
$h(\theta)\cdot T_{\lambda}^{l} =T^{l}_{\lambda}$ follows from 
$h(\theta)\cdot z=e^{i\theta}z$ for $z\in\C^{2}$.
Moreover, direct computation shows
$$
h_{1}(a)\cdot D'=D'+\ov{a}(z_{1}-a\ov{z}_{2})\frac{\p}{\p \ov{z}_{1}}-a\ov{z}_{2}\frac{\p}{\p z_{2}}
\quad (a\in\C).
$$
Hence in the same way as in $n\geq 3$, the following equality of distributions on $\C^{2}\backslash\{0\}\simeq\R^{4}\backslash\{0\}$ holds:
\begin{eqnarray*}
h_{1}(a)\cdot T^{l}_{2}(z)&=&\left(D'+\ov{a}(z_{1}-a\ov{z}_{2})\frac{\p}{\p \ov{z}_{1}}-a\ov{z}_{2}\frac{\p}{\p z_{2}}\right)^{l}\delta(z_{2},\ov{z}_{2})\\
&=&(D')^{l}\delta(z_{n},\ov{z}_{n})\\
&=&T^{l}_{2}(z).
\end{eqnarray*}
Therefore we have $T^{l}_{2}\in\D'(\C^{2}\backslash\{0\})^{H}_{even,-2}$ for any $l\in\N$.
Furthermore, we have $\dim \D'(\C^{2}\backslash\{0\})^{H}_{even,-2}=\infty$ because $\{T^{l}_{2}\}_{l\in\N}$ are linearly independent.
Thus 
$H$ satisfies 2)
by Corollary \ref{H}.
Therefore the proof of the case of $n=2$ completes.
\end{proof}
\begin{Remark}
For $n=2$, the dimension of 
$\D'(\C^{2}\backslash\{0\})^{H}_{even,-\lambda}$ is finite-dimensional
for generic $\lambda\in\C$. 
Indeed one can show that
$$
\dim \D'(\C^{2}\backslash\{0\})^{H}_{even,-\lambda}\leq 2\quad
\text{ for $\lambda\in\C\backslash\{2\}$}.
$$
\end{Remark}
Finally, we discuss the supports of elements of 
$\D'(G/Q,\chi_{\lambda})^{H}$.
If $\lambda\notin 2\N+4$, we have
${\rm supp}(T^{l}_{\lambda})=cl(Y_{2n-3})$ by (\ref{TT1}).
 Here $cl(Y_{2n-3})$ denotes the closure of $Y_{2n-3}$ in $X$
(See (\ref{Y}) for the definition of $Y_{2j-1}\subset X$ for $j\in\{1,2,\dots,n\}$ and hereafter we regard as $Y_{2j-1}\subset G/Q$ by $X\simeq G/Q$ in (\ref{XGQ})).
We put $X_{j}:={\it cl}(Y_{2j-1})\subset X$. 
Then we have 
$$\dim \left(\D'_{X_{n-1}}(G/Q,\chi_{\lambda})^{H}\middle/
\D'_{X_{n-2}}(G/Q,\chi_{\lambda})^{H}
\right)=\infty,$$
where $\D'_{X_{j-1}}(G/Q,\chi_{\lambda})^{H}:=\{F\in\D'(G/Q,\chi_{\lambda})^{H}\mid {\rm supp}(F)\subset X_{j-1}\}$.
Furthermore, the following statement holds more generally:
\begin{Proposition}
\label{suppinfty}
Suppose $n\geq 3$.
Let $G$ and $Q$ be as in Theorem \ref{Main}. 
Then for any $j\in\{2,3,\dots,n-1\}$, we have
\begin{eqnarray*}
\dim \left(\D'_{X_{j}}(G/Q,\chi_{\lambda})^{H}\middle/
\D'_{X_{j-1}}(G/Q,\chi_{\lambda})^{H}
\right)=\infty
\end{eqnarray*}
for any $\lambda\in\C\backslash(2\N+2+2n-2j)$.
\end{Proposition}
\begin{proof}
Let $D_{j}$ be a real analytic vector field on $\C^{n}\simeq\R^{2n}$ given by $D_{j}:=\ov{z}_{j-1}\frac{\p}{\p \ov{z}_{j}}+z_{j}\frac{\p}{\p z_{j+1}}$.
For $l\in\N$, we define a distribution $T^{l}_{\lambda,j}\in\D'(\C^{n}\backslash\{0\})$ with holomorphic parameter $\lambda\in\C$ by
\begin{eqnarray}
T^{l}_{\lambda,j}(z):=\frac{1}{\Gamma\left(n-j+1-\frac{\lambda}{2}\right)}
D_{j}^{l}
\left(
|z_{j}|^{2(n-j)-\lambda}
\prod_{k=j+1}^{n}
\delta(z_{k},\ov{z}_{k})
\right).
\label{TT3}
\end{eqnarray}
Then we have $T^{l}_{\lambda,j}\in\D'(\C^{n}\backslash\{0\})^{H}_{even,-\lambda}\simeq\D'(G/Q,\chi_{\lambda})^{H}$ 
in the same way as the case of $T^{l}_{\lambda}$. Moreover ${\rm supp}(T^{l}_{\lambda,j})=cl(Y_{2j-1})=X_{j}$ follows easily from $(\ref{TT3})$ if $\lambda\in\C\backslash(2\N+2+2n-2j)$. 
This completes the proof of Proposition \ref{suppinfty}.
\end{proof}
\begin{Remark}
\label{holonomic}
Let $G_{\C},Q_{\C}$ and $H_{\C}$ be complexifications of $G,Q$ and $H$, respectively.
Then if $\#(H_{\C}\backslash G_{\C}/Q_{\C})<\infty$, 
we have $\dim \D'(\C^{n}\backslash\{0\})^{H}_{even,-\lambda}<\infty$ for any $\lambda\in\C$ 
by the general theory of holonomic systems due to Sato-Kashiwara-Kawai \cite[Theorems\:5.1.7, and\:5.1.12]{K}.
Therefore we have $\#(H_{\C}\backslash G_{\C}/Q_{\C})=\infty$
because $\dim \D'(\C^{n}\backslash\{0\})^{H}_{even,-\lambda}=\infty$ by Proposition \ref{Dinfty}.
Alternatively we can show that $\#(H_{\C}\backslash G_{\C}/Q_{\C})=\infty$ by direct calculation
as below.
\begin{Proposition}
\label{Cinfty}
Suppose $G,Q$ are as in Theorem \ref{Main},
and $H$ is the subgroup of $G$ defined in (\ref{Hgroup}).
Let $G_{\C},Q_{\C}$ and $H_{\C}$ be complexifications of $G,Q$ and $H$, respectively.
Then if $n\geq 2$,
we have $\#(H_{\C}\backslash G_{\C}/Q_{\C})=\infty$.
\end{Proposition}
Before the proof of Proposition \ref{Cinfty},
we discuss the complexifications of $\C$ and $R_{\ve}$
in order to make calculation clear.
We write $\ov{\C}$ for the complex conjugate space of $\C$,
that is,
$\ov{\C}=\C$ as a set,
and scalar multiplication of $c\in\C$ given by $c\cdot v:=\ov{c}v$ for 
$v\in\ov{\C}$.
Then the complexification $\C\otimes_{\R}\C$ of $\C$
is isomorphic to $\C\oplus\ov{\C}$ 
as a $\C$-algebra by the following map:
\begin{eqnarray}
e_{-}\frac{a\otimes 1}{2}+
e_{+}\frac{c\otimes 1}{2}
&\mapsto&
(a,c)\quad (a,c\in\C),
\label{CC}
\end{eqnarray}
where $e_{\pm}:=1\otimes 1 \pm i\otimes i \in \C\otimes_{\R} \C$.
Here the multiplication of $\C\otimes_{\R}\C$ is given by
$(a\otimes b)\cdot (c\otimes d)=ac\otimes bd$.
Similarly, we define an isomorphism $R_{\ve}\otimes_{\R}\C=\left(\C\oplus\C\ve\right)\otimes_{\R}\C
\DistTo\left(\C\oplus\ov{\C}\right)\oplus\left(\C\oplus\ov{\C}\right)\ve$
as a $\C$-algebra by
\begin{eqnarray}
e'_{-}\frac{(a+b\ve)\otimes 1}{2}+
e'_{+}\frac{(c+d\ve)\otimes 1}{2}
\mapsto
(a,c)+(b,d)\ve\quad (a,b,c,d\in\C),
\label{RCC}
\end{eqnarray}
where $e'_{\pm}:=1\otimes 1 \pm i\otimes i \in (\C\oplus \C \ve)\otimes_{\R} \C$.
Then the multiplication on $\left(\C\oplus\ov{\C}\right)\oplus\left(\C\oplus\ov{\C}\right)\ve$ induced from this isomorphism  is given below,
$$
\left((a,c)+(b,d)\ve\right)
\left((a',c')+(b',d')\ve\right)
=(aa'+b\ov{d'},cc'+d\ov{b'})+(ab'+b\ov{c'},cd'+d\ov{a'})\ve
,$$
where $(a,c)+(b,d)\ve,(a',c')+(b',d')\ve\in\left(\C\oplus\ov{\C}\right)\oplus\left(\C\oplus\ov{\C}\right)\ve$.
Hereafter we identify $R_{\ve}\otimes_{\R}\C$ with $\left(\C\oplus\ov{\C}\right)\oplus\left(\C\oplus\ov{\C}\right)\ve$ via (\ref{RCC}). For the proof of Proposition \ref{Cinfty}, we need:
\begin{Lemma}
\label{ve}
The complexification of the representation of $R_{\ve}$ on $\C$ defined in (\ref{action}) 
is given below under the identifications of (\ref{CC}) and (\ref{RCC}),
$$
\left((a,c)+
(b,d)\ve\right)\cdot(z,w)
=(az,cw)+(b\ov{w},d\ov{z}),
$$
where $(a,c)+(b,d)\ve\in R_{\ve}\otimes_{\R}\C$ and $(z,w)\in\C\oplus\ov{\C}\simeq\C\otimes_{\R}\C$.
\end{Lemma}
This follows from easy calculation, hence we omit the proof.
\begin{proof}[Proof of Proposition {\rm \ref{Cinfty}}]
$M_{n}(R_{\ve}\otimes\C)$ acts on
$(\C\oplus \ov{\C})^{n}\simeq\C^{n}\otimes\C$ by left multiplication.
This action induces $\iota_{\C}:M_{n}(R_{\ve}\otimes \C)\DistTo M_{2n}(\C)$ in the same way as $\iota$ in (\ref{iota}).
Then the complexification of $H$ in $M_{n}(R_{\ve}\otimes\C)$ is the following:
\begin{eqnarray}
\hspace{-0.2cm}
H_{\C}\hspace{-0.1cm}:=\hspace{-0.1cm}
\left\{
h^{a}(\bm{A}):= \hspace{-0.1cm}\left(\begin{array}{cccc}
 \hspace{-0.2cm}(e^{ia},e^{i\ov{a}})&\hspace{-0.5cm}A_{1}\ve&\hspace{-0.2cm}\cdots \hspace{-0.1cm}&\hspace{-0.1cm}A_{n-1}\ve^{n-1}\\
 &\hspace{-0.4cm}(e^{ia},e^{i\ov{a}})&\hspace{-0.5cm}\ddots \hspace{-0.5cm}&\hspace{-0.2cm}\vdots\\
 &&\hspace{-0.5cm}\ddots\hspace{-0.5cm} &\hspace{-0.2cm}A_{1}\ve\\
 &&&\hspace{-0.1cm}(e^{ia},e^{i\ov{a}})
 \end{array}\hspace{-0.1cm}
\right)
\middle|\:\hspace{-0.1cm}
\begin{array}{c}
a\in\C\\
\hspace{-0.1cm}\bm{A}\in (\C\oplus\ov{\C})^{n-1}\hspace{-0.2cm}
\end{array}
\right\}\hspace{-0.1cm},
\label{HC}
\end{eqnarray}
where $\bm{A}=(A_{1},\dots,A_{n-1})\in (\C\oplus \ov{\C})^{n-1}$.
Similarly to the case of $H$ in $(\ref{Hgroup})$, $\iota(H_{\C})$ is a subgroup of $G_{\C}=SL(2n,\C)$ and we may identify $H_{\C}$ in $M_{n}(R_{\ve}\otimes\C)$ with
$\iota_{\C}(H_{\C})$ in $G_{\C}=SL(2n,\C)$. 
Let $\C^{\times}:=GL(1,\C)$ act on $\left(\C\oplus\ov{\C}\right)^{n}$
by scalar multiplication. Then, for $c\in\C^{\times}$ and 
$\left((z_{1},w_{1}),\dots,(z_{n},w_{n})\right)\in(\C\oplus\ov{\C})^{n}$, we have
$$c\cdot\left((z_{1},w_{1}),\dots,(z_{n},w_{n})\right)=\left((cz_{1},\ov{c}w_{1}),\dots,(cz_{n},\ov{c}w_{n})\right).$$
We put 
$X_{\C}:=\left(\left(\C\oplus\ov{\C}\right)^{n}\backslash\{0\}\right)/\C^{\times}$.
By regarding $(\C\oplus\ov{\C})^{n}$ as $\C^{2n}$, we have
$X_{\C}\simeq \C{\mathbb P}^{2n-1}\simeq G_{\C}/Q_{\C}$ and these isomorphisms induce a bijection:
$$H_{\C}\backslash X_{\C}\simeq H_{\C}\backslash G_{\C}/Q_{\C}.$$
On the other hand, the action of $H_{\C}$ on $(\C\oplus\ov{\C})^{n}$ is given below by Lemma \ref{ve} (See (\ref{HC}) for the definition of $h^{a}(\bm{A})\in H_{\C}$),
\begin{eqnarray*}
\hspace{-0.5cm}h^{a}(\bm{A})\cdot\begin{pmatrix}
(z_{1},w_{1})\\
\vdots\\
(z_{n-1},w_{n-1})\\
(z_{n},w_{n})\end{pmatrix}
\hspace{-0.4cm}&=&\hspace{-0.4cm}\begin{pmatrix}
(e^{ia}z_{1},e^{i\ov{a}}w_{1})+
\sum_{j=1}^{n-1}(a_{j},b_{j})\ve^{j}\cdot (z_{j+1},w_{j+1})
)\\
\vdots\\
\begin{aligned}
&(\:\:e^{ia}z_{n-1}+a_{1}\ov{w}_{n}&,&\:\:e^{i\ov{a}}w_{n-1}+b_{1}\ov{z}_{n}\hspace{-0.3cm}&)\\
&(\qquad e^{ia}z_{n}&,&\qquad e^{i\ov{a}}w_{n}&)
\end{aligned}
\end{pmatrix}
\hspace{-0.13cm},
\end{eqnarray*}
where $a\in\C$, $\bm{A}=(A_{1},\dots,A_{n-1})=\left((a_{1},b_{1}),\dots,(a_{n-1},b_{n-1})\right)\in \left(\C\oplus\ov{\C}\right)^{n-1}$ and  
$\left((z_{1},w_{1}),\dots,(z_{n},w_{n})\right)\in(\C\oplus\ov{\C})^{n}$.
For $\zeta\in\C$, we define a complex $(2n-3)$-dimensional submanifold of $X_{\C}$ by
$$Y^{\zeta}_{2n-3}:=\{(z_{j},w_{j})_{j=1}^{n}\in (\C\oplus\ov{\C})^{n}\mid w_{n}=0,\: z_{n}\neq0,\: z_{n-1}=\zeta z_{n}\}/\C^{\times}\:\:\subset X_{\C}.$$
Then for any $\zeta\in\C$, the group $H_{\C}$ leaves $Y^{\zeta}_{2n-3}$ invariant, and in fact it acts transitively.
Moreover if $\zeta\neq \mu$, $Y^{\zeta}_{2n-3}$ and $Y^{\mu}_{2n-3}$ have no intersection.
Therefore we have $\#(H_{\C}\backslash G_{\C}/Q_{\C})=\#(H_{\C}\backslash X_{\C})= \infty$.
\end{proof}
\end{Remark}
\section*{Ackowledgement}
The author is grateful to Professor Toshiyuki Kobayashi for his much helpful advice and constant encouragement and thanks my parents for their support.

\end{document}